\documentclass[preprint,sort&compress,12pt]{elsarticle}
\usepackage{amscd,amssymb,amsmath,amsthm}
\usepackage[all]{xy}
\usepackage{array}
\usepackage{etoolbox}
\usepackage{mathtools}
\DeclarePairedDelimiter{\ceil}{\lceil}{\rceil}

\makeatletter
\patchcmd{\ps@pprintTitle}{\footnotesize\itshape
       Preprint submitted to \ifx\@journal\@empty Elsevier
       \else\@journal\fi\hfill\today}{\relax}{}{}
\makeatother

\newdir{ >}{!/8pt/\dir{}*\dir{>}}

\newtheorem{theorem}{Theorem}[section]

\newtheorem{lemma}[theorem]{Lemma}
\newtheorem{corollary}[theorem]{Corollary}

\theoremstyle{definition}

\numberwithin{equation}{theorem}

\DeclareMathOperator{\M}{\mathit{M(G)}}
\DeclareMathOperator{\MM}{\mathit{M}}
\DeclareMathOperator{\e}{\mathit{exp}}

\DeclareMathOperator{\Z}{\mathit{Z(G)}}

\DeclareMathOperator{\p}{\mathit{p}}
\DeclareMathOperator{\cc}{\mathit{c}}

\DeclareMathOperator{\G}{\mathit{G}}
\DeclareMathOperator{\N}{\mathit{N}}

\begin{document}

\begin{frontmatter}

\title{A lemma on the exponent of Schur multiplier of $p$ groups with good power structure}

 \author[IISER TVM]{A.E. Antony}
\ead{ammu13@iisertvm.ac.in}
\author[IISER TVM]{P. Komma}
\ead{patalik16@iisertvm.ac.in}
\author[IISER TVM]{V.Z. Thomas\corref{cor1}}
\address[IISER TVM]{School of Mathematics,  Indian Institute of Science Education and Research Thiruvananthapuram,\\695551
Kerala, India.}
\ead{vthomas@iisertvm.ac.in}
\cortext[cor1]{Corresponding author. \emph{Phone number}: +91 8078020899.}

\begin{abstract}

In this note, we give short proofs of the well-known results that the exponent of the Schur multiplier $\M$ divides the exponent of $\G$ for finite $\p$-groups of maximal class and potent $\p$-groups. Moreover, we prove the same for a finite $\p$-group $\G$ satisfying $\G^{\p^2}\subset \gamma_{\p}(\G)$, and for $3$-groups of class $5$. We do this by proving a general lemma, and show that these three classes of groups satisfy the hypothesis of our lemma. 
\end{abstract}

\begin{keyword}

\MSC[2010]   20B05 \sep 20D10 \sep 20D15 \sep 20F05 \sep 20F14 \sep 20F18 \sep 20G10 \sep 20J05 \sep 20J06 

\end{keyword}

\end{frontmatter}

\section{A simple lemma}

We begin with the lemma mentioned in the abstract. 
\begin{lemma}\label{L:1.1}
Let $\p$ be an odd prime and $\G$ be a finite $\p$-group with $\e(\G)=\p^n$. Suppose $\G$ satisfies the following properties:
\begin{itemize}
\item [$(i)$] $\G^{\p}$ is powerful
\item[$(ii)$] $\e(\G^{\p})$ is $\p^{n-1}$
\item [$(iii)$] $\gamma_{\p+1}(\G)\subset \G^{\p}$
\end{itemize}
Then $\e(\G\wedge \G)\mid \e(\G)$. In particular, $\e(\M)\mid \e(G)$.

\end{lemma}

\begin{proof}
Consider the following exact sequence $$\G^{\p}\wedge \G\rightarrow \G\wedge \G \rightarrow \frac{\G}{\G^{\p}}\wedge \frac{\G}{\G^{\p}} \rightarrow 1,$$ which yields $\e(\G\wedge \G) \mid \e(\G^{\p} \wedge \G)\e(\G/\G^{\p} \wedge \G/\G^{\p})$. $\G^{\p}$ being powerful, $\e(\G^{p} \wedge \G) \mid \e(\G^{\p})=\p^{n-1}$ by Theorem 5.2 of \cite{APT}. Note that $\G/\G^{\p}$ has class at most $\p$. Now applying Theorem 3.11 of \cite{APT}, we obtain $\e(\G/\G^{\p} \wedge \G/\G^{\p}) \mid \p$ finishing the proof.
\end{proof}
  
 The groups satisfying $\gamma_{\p}(\G)\subset \G^{\p^2}$ were considered by L. E. Wilson in \cite{LEW}. The class of groups satisfying $\gamma_{m}(\G)\subset \G^{\p}$ for $2\le m\le \p-1$ were considered by D. Arganbright in \cite{DA}, and it includes the class of potent $\p$-groups considered by J. Gonzalez-Sanchez and  A. Jaikin-Zapirain in \cite{SZ}. The author of \cite{PM6} proves that  $\e(\G \wedge \G)\mid \e(\G)$ for a $\p$-group $\G$ of maximal class. In \cite{PM4}, the author proves $\e(\M)\mid\e(\G)$ for potent $\p$-groups. We do the same in the next corollary, and moreover we include one more class of groups
 
 \begin{corollary}
Let $\G$ be a finite $\p$-group satisfying any of the following conditions: 
\begin{itemize}
\item[$(i)$] (Moravec, \cite{PM6}) $\G$ is a $\p$-group of maximal class
 \item[$(ii)$] (Moravec, \cite{PM4}) $\G$ is a potent $\p$-group
\item[$(iii)$] $\gamma_{\p}(\G)\subset \G^{\p^2}$
\end{itemize}
Then $\e(\G\wedge \G)\mid \e(\G)$. In particular, $\e(\M)\mid \e(G)$.
\end{corollary}

\begin{proof}
\begin{itemize}
\item[$(i)$] Suppose $\G$ is a $2$-group of maximal class, then there exist a cyclic normal subgroup of index $2$, and hence $\G$ is a metacyclic group, and the claim holds by Proposition $7.5$ of \cite{AT}. If $|\G|\leq \p^{\p+1}$, then the nilpotency class of $\G$ is less than or equal to $\p$. So by Theorem 3.11 of \cite{APT}, we obtain that $\e(\G\wedge \G)\mid \e(\G)$. So we may assume that $|\G|\geq \p^{\p+2}$. In this case, consider the fundamental subgroup $\G_1$ of $\G$, defined as the centralizer of $\gamma_2(\G)/\gamma_4(\G)$ in $\G$. By Theorem 9.6 ($b$) and Theorem 9.8 ($a$) of \cite{BJ3}, it follows that $\G_1$ is regular. Moreover $\gamma_{\p}(\G) = \G_1^{\p} = \G^{\p}$ (cf. Theorem 9.6 ($a$) of \cite{BJ3} and Exercise 2, pg. 119 of \cite{YB1}). Now apply Lemma \ref{L:1} to obtain the result.

\item[$(ii)$] $(iii)$ Let $\G$ satisfy the condition $\gamma_{\p}(\G)\subset \G^{\p^2}$ or $\G$ be a potent $\p$-group. Then it follows from \cite{LEW}, \cite{SZ} that $\G$ satisfies the three conditions of Lemma $\ref{L:1}$, and hence the result.
\end{itemize}
\end{proof}

\begin{corollary}
Let $\p$ be an odd prime and $\G$ be a $\p$-group of coclass $r$. Keeping the notations and assumptions of Theorem 4.7  of \cite{PM6}, we have $\e(G\wedge G) \mid p^{e+fn}$, where $n = 1+ \ceil{log_{p-1}(\frac{m}{p+1})}$.
\end{corollary}
\begin{proof}
Consider the following exact sequence, $\gamma_m(\G) \wedge \G \rightarrow \G\wedge \G \rightarrow \G/\gamma_m(\G) \wedge \G/\gamma_m(\G) \rightarrow 1,$ which yields $\e(\G\wedge \G) \mid \e(\gamma_m(\G) \wedge \G)\e(\\\G/\gamma_m(\G) \wedge \G/\gamma_m(\G))$. Now $\gamma_m(\G)$ being powerful (cf. Theorem 1.2 of \cite{AS}) and class of $\G/\gamma_m(\G)$ being at most $m-1$, applying Theorem 5.2 and Theorem 6.5 of \cite{APT} yields the required bound.
\end{proof}

\section{$3$-groups of class $5$.}
We follow right notations: $x\ ^y=y^{-1}xy$, $[x, y]=x^{-1}x^y$ and the commutators are left normed. We denote the $i$-th terms of the lower and upper central series of $\G$ by $\gamma_i(\G)$ and $Z_i(\G)$, respectively.
We recall the commutator identities
\begin{equation}\label{eq:1.1}
[xy, z] = [x, z]^y [y, z],
\end{equation}
\begin{equation}\label{eq:1.2}
[z, xy] = [z, y] [z, x]^y.
\end{equation}
A standard argument shows that, in a group $\G$, any commutator of weight $r$ is multilinear modulo $\gamma_{r+1}(\G)$, in particular
\begin{align*}
[g_1, \dots, x_iy_i, \dots, g_r]\equiv [g_1, \dots, x_i, \dots, g_r] [g_1, \dots, y_i, \dots, g_r]\ \text{mod}\ \gamma_{r+1}(\G)
\end{align*}
 for all $i=1, \dots, r$. We recall the collection formulae of P. Hall given in Proposition $1.1.32$ of \cite{LGM}.
 \begin{theorem}[Commutator collection formulae]\label{Hall}
 Let $x$ and $y$ be elements of $\G$, and let $\p$ be a prime and $r$ a positive integer. For $x, y\in \G$ define $K(x, y)$ be the normal closure in $\G$ of the set of all commutators in $x, y$ of weight at least $\p^r$ and of weight at least $2$ in $y$, together with the $\p^{r-k+1}$-th powers of of all the basic commutators in $\{x, y\}$ of weight less than $p^k$ and weight at least $2$ in $y$ for $1\le k\le r$. Then
 \begin{itemize}
\item [$(i)$] $(xy)^{\p^r}\equiv x^{\p^r} y^{\p^r} [y, x]^{{\p^r}\choose{2}} [y,\ _2 x]^{{\p^r}\choose{3}} \cdots [y,\ _{\p^r-1}x]\ \text{mod}\ K(x, y)$
\item [$(ii)$] $[x^{\p^r}, y]\equiv [x, y]^{\p^r} [x, y, x]^{{\p^r}\choose{2}} \cdots [x, y,\ _{\p^r-1}\ x]\ \text{mod}\ K(x, [x, y])$
 \end{itemize}
 \end{theorem}

The $\p$-groups of class at most $\p-1$ are an example of regular $\p$-groups, which have a nice power-commutator structure. A $\p$-group of class $\p$ need not be regular, but they share a nice property as discovered by A. Mann \cite{AM}.
 \begin{lemma}\label{Mann} $($Mann$)$
 Let $\G$ be a $\p$-group of class $\cc\le \p$, and let $x, y\in \G$. Then $[x, y^{\p^n}]=1$ is equivalent to $[x, y]^{\p^n}=1$ and, similarly, it is equivalent to $[x^{\p^n}, y]=1$.
 \end{lemma}
 We remark that the original version of this lemma is stated for $n=1$, nonetheless the case $n>1$ requires a minor modification, namely the use of [Theorem $7.2(a)$, \cite{YB1}].

 For the reader's convenience, we recall the following expansion in a nilpotent group of class $5$ given on page $495$ of \cite{PM5}.

\begin{align}\label{eq:class5}
\notag(x y)^n= &x^ny^n[y, x]^{{n}\choose{2}} [y, x, x]^{{n}\choose{3}} [y, x, y]^{{{n}\choose{2}}+2{{n}\choose{3}}} [y, x, x, x]^{{n}\choose{4}}\\
\notag&[y, x, x, y]^{2{{n}\choose{3}}+3{{n}\choose{4}}} [y, x, y, y]^{2{{n}\choose{3}}+3{{n}\choose{4}}} [y, x, x, [y, x]]^{{{n}\choose{3}}+7{{n}\choose{4}}+6{{n}\choose{5}}}\\
\notag&[y, x, y, [y, x]]^{6{{n}\choose{3}}+18{{n}\choose{4}}+12{{n}\choose{5}}} [y, x, x, x, x]^{{n}\choose{5}}\\
&[y, x, x, x, y]^{3{{n}\choose{4}}+4{{n}\choose{5}}} [y, x, x, y, y]^{{{n}\choose{3}}+6{{n}\choose{4}}+6{{n}\choose{5}}} [y, x, y, y, y]^{3{{n}\choose{4}}+4{{n}\choose{5}}}.
\end{align}

Using the above, we obtain the following identities.

\begin{lemma}\label{L:1}
Let $\G$ be a nilpotent group of class $5$ and $n$ be a positive integer.
\begin{itemize}
\item [$(i)$] for all $x\in \gamma_2(\G)$ and $y$ in $\G$ we have
\begin{align*}
(yx)^n= &y^nx^n[x, y]^{{n}\choose{2}} [x, y, y]^{{n}\choose{3}} [x, y, x]^{{{n}\choose{2}}+2{{n}\choose{3}}} [x, y, y, y]^{{n}\choose{4}}.
\end{align*}
\item [$(ii)$] for all $x, y\in \G$ we have 
\begin{align*}
[y^n, x]= &[y, x]^n [y, x, y]^{{n}\choose{2}} [y, x, y, y]^{{n}\choose{3}} [y, x, y, [y, x]]^{{{n}\choose{2}}+2{{n}\choose{3}}} [y, x, y, y, y]^{{n}\choose{4}}.
\end{align*}
\end{itemize}
\end{lemma}

The next lemma provides the first step in proving that a $3$-group of class $5$ satisfies the hypothesis of Lemma \ref{L:1.1}.
\begin{lemma}\label{L:2}
Let $\G$ be a $3$-group of class $5$ and $n\ge 2$. If $\exp(\G)=3^n$, then $\exp(\G^3)=3^{n-1}$.
\end{lemma}
\begin{proof}
To prove $\exp(\G)=3^{n-1}$, we prove $(x^3y^3)^{3^{n-1}}=1$ for all $x, y\in \G$. We expand $(x^3y^3)^{3^{n-1}}$ using \eqref{eq:class5}, and prove that all the powers of basic commutators in the expansion are trivial. We first show $[y^3, x^3]^{3^{n-1}}=1$, which implies $[y^3, x^3]^{{3^{n-1}}\choose{2}}=1$.
To prove that $[y^3, x^3]^{3^{n-1}}=1$, we expand $[(y^3)^{3^{n-1}}, x^3]$ according to Lemma \ref{L:1} $(ii)$ and we observe that all the powers of commutators appearing are trivial. We begin with $[y^3, x^3,\ _3\ y^3] ^{{3^{n-1}}\choose{4}}$. Since $\gamma_6(\G)=1$ the commutators of weight $5$ are multilinear, so that 
\begin{equation}\label{eq:2.1}
[y^3, x^3,\ _3\ y^3] ^{3^{n-2}}= [y, x,\ _3\ y]^{3^{n+3}}= 1.
\end{equation}
In particular, since $3^{n-2}$ divides ${{3^{n-1}}\choose{4}}$ we have $[x^3,\ _4\ y^3] ^{{3^{n-1}}\choose{4}}=1$. A similar argument proves $[y^3, x^3, y^3, [y^3, x^3]]^{3^{n-2}}=1$ and, in particular, $[y^3, x^3, y^3, [y^3, x^3]]^{{{n}\choose{2}}+2{{n}\choose{3}}}=1$. Next we consider $[y^3, x^3,\ _2\ y^3]^{{3^{n-1}}\choose{3}}$. Using corollary $1.1.7$ of \cite{LGM}, we have
\begin{equation}\label{eq:2.2}
[y^3, x^3, y^3]\equiv [y^3, x^3, y]^3 [y^3, x^3, y, y]^3\ \text{mod}\ \gamma_5(\G).
\end{equation}
Using \eqref{eq:2.2} in $[[y^3, x^3, y^3], y^3]$ and expanding $[[y^3, x^3, y]^3[y^3, x^3, y, y]^3, y^3]$ by \eqref{eq:1.1}, we obtain
\begin{equation*}
[[y^3, x^3, y^3], y^3]= [y^3, x^3, y, y^3]^3 [y^3, x^3, y, y, y^3]^3.
\end{equation*}
Furthermore, $[y^3, x^3, y, y, y^3]\in \Z$ yielding
\begin{equation}\label{eq:2.3}
[y^3, x^3,\ _2\ y^3]^{3^{n-2}}= [y^3, x^3, y, y^3]^{3^{n-1}} [y^3, x^3, y, y, y^3]^{3^{n-1}}.
\end{equation}
Since $[y^3, x^3, y, y, y^3]$ has weight $5$, by multilinearity $[y^3, x^3, y, y, y^3]^{3^{n-1}}= [y, x,\ _3\ y]^{3^{n+2}}=1$. Moreover $\langle [y^3, x^3, y], y\rangle$ has class at most $3$, and $[[y^3, x^3, y], (y^3)^{3^{n-1}}]=1$, so $[y^3, x^3, y, y^3]^{3^{n-1}}=1$ by Lemma \ref{Mann}. Therefore $[y^3, x^3,\ _2\ y^3]^{{3^{n-1}}\choose{3}}=1$ by \eqref{eq:2.3}. Next to prove $[y^3, x^3, y^3]^{{3^{n-1}}\choose{2}}=1$, we see it's inverse $[y^3, [y^3, x^3]]$ has order $3^{n-1}$. Since $\langle y^3, [y^3, x^3]\rangle$ has class at most $4$ and $[(y^3)^{3^{n-1}}, [y^3, x^3]]=1$, by Lemma $3.6$ of \cite{APT}, we have $[y^3, [y^3, x^3]]^{3^{n-1}}= [y^3, x^3,\ _3\ y^3]^{{{3^{n-1}}\choose{3}}}$. Therefore, we obtain $[y^3, [y^3, x^3]]^{3^{n-1}}=1$ by \eqref{eq:2.1}. Hence from the expansion $[(y^3)^{3^{n-1}}, x^3]$, we obtain $[y^3, x^3]^{3^{n-1}}=1$.
\par Now we proceed to show $[y^3, x^3, z^3]^{3^{n-2}}=1$ for all $z\in \G$. Using corollary $1.1.7$ of \cite{LGM}, we obtain $[y^3, x^3, z^3]=[y^3, x^3, z]^3 [y^3, x^3, z, z]^3 [y^3, x^3, z, z, z]$.
Observe that $[y^3, x^3, z]$, $[y^3, x^3, z, z]$, $[y^3, x^3, z, z, z]$ commute with one another, and hence
\begin{equation}\label{eq:2.4}
[y^3, x^3, z^3]^{3^{n-2}} = [y^3, x^3, z]^{3^{n-1}} [y^3, x^3, z, z]^{3^{n-1}} [y^3, x^3, z, z, z]^{3^{n-2}}.
\end{equation}
Since $\gamma_6(\G)=1$, by multilinearity $[y^3, x^3, \ _3\ z]^{3^{n-2}}=[y, x\ _3\ z]^{3^n}=1$. Next we consider $[y^3, x^3, z]^{3^{n-1}}$. Since $\langle [y^3, x^3], z\rangle$ has class at most $4$, and $[y^3, x^3]^{3^{n-1}}=1$, we obtain $[y^3, x^3, z]^{3^{n-1}}= [z,\ _3\ [y^3, x^3]]^{{3^{n-1}}\choose{3}}=1$ by Lemma $3.4$ of \cite{APT}. Furthermore, expanding $[[y^3, x^3, z]^{3^{n-1}}, z]$ by \eqref{eq:1.1} yields that $[y^3, x^3, z, z]^{3^{n-1}}= 1$. Hence $[y^3, x^3, z^3]^{3^{n-2}}=1$ by \eqref{eq:2.4}. In particular, $[y^3, x^3, x^3]^{3^{n-2}}=1$ and $[y^3, x^3, y^3]^{3^{n-2}}=1$. Moreover $[y^3, x^3, x^3]$, $[y^3, x^3, y^3]\in Z_3(\G)$, so by Lemma \ref{Mann}, every commutator in $\G$ having weight $1$ either in $[y^3, x^3, x^3]$ or in $[y^3, x^3, y^3]$ has order at most $3^{n-2}$. Thus we obtain that all the powers of the basic commutators having weight $4$ or more in the expansion of $(x^3y^3)^{3^{n-1}}$ are trivial, yielding $(x^3y^3)^{3^{n-1}}=1$.
\end{proof}

Suppose $\N, \MM\unlhd\G$, then Hall's collection formulae yields
\begin{equation}\label{eq:Hall}
[\N^{\p}, \MM]\le [\N, \MM]^{\p} [\MM,\ _{\p}\N].
\end{equation}

For $\p$-groups of class $5$ we have that the exponent of $\G\wedge \G$ divides the exponent of $\G$ for all odd primes with the exception of $\p=3$, and Theorem \ref{th:3} of this paper is aimed to fill this gap.

\begin{theorem}\label{th:3}
If $\G$ be a $3$-group of class at most $5$, then $\e(\G\wedge\G)$ divides $\e(\G)$.
\end{theorem}
\begin{proof}
Suppose $\exp(\G)=3^n$. The case $n=1$ is Proposition 7 $(iii)$ of \cite{PM2}, so assume $n\ge 2$. By Lemma \ref{L:2}, we have $\exp(\G^3)=3^{n-1}$. Recall groups of exponent $3$ have class at most $3$ \cite{LBLW}, we obtain $\gamma_4(\G)\le \G^3$. Moreover applying \eqref{eq:Hall} yields 
$[\G^3, \G^3]\le [\G^3, \G]^3 [\G^3,\ _3\ \G]$, and since $\gamma_6(\G)=1$, $[\G^3,\ _3\ \G]\le \gamma_4(\G)^3\le (\G^3)^3$ by \eqref{eq:Hall}. Hence we have that $\G^3$ is powerful,  and thus $\exp(\G\wedge \G)\mid \exp(\G)$, by Lemma $\ref{L:1.1}$.
\end{proof}
We have that for $\p\ge 5$ the $\p$-groups of class $5$ satisfy the statement of the above theorem [Corollary 3.11, \cite{APT}]. In particular, we have the following:
\begin{corollary}\label{cr:3.8}
Let $\p$ be an odd prime. If $\G$ is a $\p$-group of class at most $5$, then $\e(\G\wedge\G)$ divides $\e(\G)$.
\end{corollary}

\section*{References}
\bibliographystyle{amsplain}
\bibliography{Bibliography}
\end{document}